\documentclass[bezier,amstex, oneside,reqno]{amsart}
\usepackage{amsmath, amssymb, amsthm, verbatim, euscript, amscd}
\usepackage[dvips]{graphicx}
\usepackage{epsfig}
\usepackage{color}
\usepackage{tocvsec2}
\usepackage[all,cmtip]{xy}

\def\ie{{i.e., }}
\def\eg{{e.g., }}

\def\R{\mathbb R}
\def\Z{\mathbb Z}

\def\U{\mathcal U}
\def\V{\mathcal V}
\def\D{\mathcal D}
\def\Dc{\mathcal D^c}
\def\H{\mathcal H}
\def\Hyx{\mathcal H_{y\to x}}
\def\W{\mathcal W}
\def\Ws{\mathcal W^s}
\def\Wu{\mathcal W^u}
\def\Wc{\mathcal W^c}
\def\Es{E^s}
\def\Eu{E^u}
\def\Ec{ E^c}
\def\P{\EuScript P}
\def\A{\EuScript A}
\def\B{\EuScript B}
\def\K{\EuScript K}
\def\E{\EuScript C}
\def\ss{\EuScript S}
\def\T{\mathbb T}
\def\e{\varepsilon}

\def\diff-sm{diffeomorphism}
\def\nbhd{neighborhood }
\def\R{Riemannian }
\def\ph{partially hyperbolic }
\def\phsp{partially hyperbolic skew products }
\def\homeo{homeomorphism }
\def\dc{dynamically coherent }

\newtheorem*{q}{Main Question}
\newtheorem*{TC}{Main Theorem}
\newtheorem{prop}{Proposition}[section]
\newtheorem{question}[prop]{Question}
\newtheorem{theorem}[prop]{Theorem}

\newtheorem*{claim}{Claim}

\newtheorem{cor}[prop]{Corollary}
\newtheorem{lemma}[prop]{Lemma}
 \theoremstyle{remark}
\newtheorem{remark}[prop]{Remarks}
\newtheorem{rmk}[prop]{Remark}

\numberwithin{equation}{section}

\begin{document}
\author{Andrey Gogolev$^\ast$}
\title[Partially hyperbolic diffeomorphisms]{Partially hyperbolic diffeomorphisms with 
compact center foliations}
\thanks{$^\ast$This research was supported in part by the NSF grant DMS-1001610}
\begin{abstract}
Let $f\colon M\to M$ be a partially hyperbolic diffeomorphism such that all of its center leaves are compact.
We prove that Sullivan's example of a circle foliation that has arbitrary long leaves cannot be 
the center foliation of $f$.
This is proved by thorough study of the accessible boundaries of the center-stable and the center-unstable leaves.

Also we show that a finite cover of $f$ fibers over an Anosov toral automorphism if one of the following
conditions is met:
\begin{enumerate}
 \item the center foliation of $f$ has codimension 2, or
\item the center leaves of $f$ are simply connected leaves and the unstable foliation of $f$ is one-dimensional.
\end{enumerate}

\end{abstract}
\date{}
 \maketitle

\tableofcontents
\settocdepth{section}
\section{Introduction}

\subsection{}

Let $M$ be a smooth compact manifold without boundary. A \diff-sm
$f\colon M\to M$ is called {\it \ph}if there exists a \R metric on $M$ along with 
a $Df$-invariant continuous splitting $TM=\Es\oplus
\Ec\oplus\Eu$ of the tangent bundle such that for all $x\in M$ all unit vectors $v^\sigma \in E^\sigma(x)$,
 $\sigma=s, c, u,$ satisfy the following properties
\begin{multline*}
\qquad\|D_xf(v^s)\|< 1, \\
\qquad\shoveleft{\|D_xf(v^u)\|> 1, \hfill}\\
\qquad\shoveleft{\|D_xf(v^s)\|<\|D_xf(v^c)\|<\|D_xf(v^u)\|.\hfill}
\end{multline*}

It is known that there are $f$-invariant foliations $\Ws$ and $\Wu$,
the {\it stable} and {\it unstable} foliations, that are tangent to
the distributions $\Es$ and $\Eu$, respectively. In general, the
center distribution $\Ec$ does not necessarily integrate to a
foliation. In this paper we will study \ph diffeomorphisms whose {\it center
distribution $\Ec$ uniquely integrates to a foliation $\Wc$ with
compact leaves}. Foliations with all leaves compact are called {\it
compact}.

The following question was posed by Charles Pugh (Problem 48
in~\cite{RHRHU}).

\begin{q} Consider a \ph \diff-sm $f$ that has a compact
center foliation. Is it true that the volume of the center leaves is
uniformly bounded? Is it true that $f$ can be finitely covered by a
\ph \diff-sm $\tilde f \colon \widetilde M \to \widetilde M$ so that
there is a fibration $p \colon \widetilde M \to N$ whose fibers are
the center leaves and an Anosov \diff-sm $\bar f \colon N\to N$
such that $p$ is a semiconjugacy between $\tilde f$ and $\bar f$?
\end{q}
 We
have modified the question from~\cite{RHRHU} to accommodate some simple examples with
finite covers.

An affirmative answer to the Main Question would reduce the classification problem of \ph
diffeomorphisms with compact center foliations to the classification
problem of Anosov diffeomorphisms.
This paper provides positive results for \ph
diffeomorphisms under each of the following additional
assumptions:
\begin{enumerate}
\item[1)] the center foliation $\Wc$ has codimension $2$;
\item[2)] the leaves of $\Wc$ are simply connected;
\item[3)] $\dim E^c=1$, $\dim E^s \le 2$, $\dim E^u \le 2$.
\end{enumerate}

\subsection{} We can place the Main Question into a wider framework of the classification
problem. In general, the classification of \ph diffeomorphisms beyond dimension 3
seems to be a hopeless problem. However one can try to classify (dynamically coherent) \ph 
diffeomorphisms based on the properties of the center foliation.

The leaves of the stable and unstable foliations are all homeomorphic to Euclidean balls. The 
center foliation however displays a variety of behaviors and one can pose classification
questions under various assumptions on the center foliation. From this perspective the Main 
Question is probably the first question that comes to mind.
\begin{question}
 Assume that the center foliation $\Wc$ is a foliation with all leaves diffeomorphic to $\mathbb R$.
Is it true that a finite cover of $f$ is center conjugate to an affine map of $G/\Gamma$, 
$x\Gamma\mapsto g \cdot A(x)\Gamma$?
\end{question}
\begin{question}
 Assume that $\dim\Wc=1$ and that $\Wc$ has countable many leaves that
are diffeomorphic to $S^1$ and the rest diffeomorphic to $\mathbb R$. Is it
true that a finite cover of $f$ is center conjugate to a time-1 map of an Anosov flow?
\end{question}
\begin{question}
 Assume that $\dim\Wc=1$ and that the map induced on the space of center leaves is identity. Is it
true that a finite cover of $f$ is center conjugate to a time-1 map of an Anosov flow?
\end{question}
For positive results one may
start by introducing additional assumptions such as $\dim E^u=1$.

\subsection{Foliation theory perspective}
Compact foliations were studied for their own sake. Below is background to help better understand the Main Question.

\subsubsection{}\label{section_vol}
Given a compact foliation $\W$ of a compact manifold $M$ we can form
the {\it leaf space} $X$ by collapsing each leaf to a point. We
equip $X$ with the quotient topology, which makes it a compact
space.

A \R metric induces a \R volume on the leaves of $\W$ and hence a
volume function $vol\colon X \to (0, +\infty)$. Exact values of
$vol$ depend on the choice of the \R metric. However, the property
of $vol$ being bounded (or unbounded) is independent of such a
choice.

\begin{prop} \label{prop_hausdorff} 
Function $vol\colon X \to (0, +\infty)$ is lower semi-continuous and
it is continuous on an open dense subset of $X$. Function $vol$ is
bounded if and only if $X$ is Hausdorff.
\end{prop}

For a proof see, for example,~\cite{Ep, E}.  The proof of semi-continuity 
given in~\cite{Ep} is for the case $\dim M=3,\, \dim\W=1$. However, the argument can be 
adjusted to the higher dimensional setup with only minor modifications.

\begin{rmk}
 We will write $length$ instead
of $vol$ when $\dim\W=1$.
\end{rmk}

\subsubsection{} It is a corollary of the Reeb Stability Theorem
that a foliation with simply connected leaves is a fibration and, hence,
its $vol$ is bounded.

\subsubsection{}
It is not known whether a compact foliation with leaves of
arbitrarily large volume can have simply connected leaves (see
Question A.1.2 in~\cite{L}).

\subsubsection{}\label{sec_codim_2}
The Main Question is already very interesting in the case when the
center foliation is a foliation by circles. Sullivan~\cite{S} gave
 a beautiful example of a smooth foliation by circles of a
compact 5-manifold with arbitrarily long circles (hence the function $length$ is
unbounded). Later similar example was constructed~\cite{EV} on a
compact 4-manifold. This is optimal since Epstein~\cite{Ep} had
shown that lengths of the circles that foliate a 3-manifold are
uniformly bounded. Also Vogt~\cite{V} and Edwards, Millet,
Sullivan~\cite{EMS} had independently generalized Epstein's result
to codimension 2 foliations.

\subsection{Seifert fibrations}\label{Seifert}
Let $\W$ be a compact foliation on $M$. Associated to every leaf $\W(x)$ is the group $G_x(\W)$ of germs 
of holonomy homeomorphisms of a small transversal centered at $x$. We
say that $\W$ is a {\it Seifert fibration} if there are finitely many {\it exceptional
fibers} $\W(x)$ (leaves of $\W$) for which $G_x(\W)$ is a non-trivial finite group
isomorphic to a subgroup of the orthogonal group $O(\dim M-\dim \W)$ and the other leaves
have trivial holonomy group.

We also note that in the case when $G_x(\W)$ is finite it can actually be identified with a
group of homeomorphisms of a transversal about $x$.

\subsection{Partially hyperbolic skew products}
Let $X$ be a compact {\it topological manifold}, that is, a
Hausdorff second countable metric space with locally Euclidean
structure given by continuous charts. Let $M$ be a compact smooth
manifold and $p\colon M\to X$ be a locally trivial fibration such
that $\Wc\stackrel{\textup{def}}{=}\{p^{-1}(x), x\in X\}$ is a continuous
foliation with $C^1$ leaves (see Section~\ref{section_foliation} for the 
definition). A $C^1$ \ph \diff-sm of $M$ with
$\Ec\stackrel{\textup{def}}{=}T\Wc$ being the center distribution is called
{\it \ph skew product}.
Hence a \ph skew product $f$ fits into the commutative diagram:
$$
\begin{CD}
 M@>f>>M\\
@VpVV@VpVV\\
X@>\bar f>>X
\end{CD}
 $$
We say that $f$ {\it fibers over a hyperbolic automorphism} if the induced
homeomorphism $\bar f\colon X\to X$ is topologically conjugate to a 
hyperbolic automorphism.

The basic examples of \phsp are skew products over Anosov
diffeomorphisms with fiber expansion/contraction dominated by the
base expansion/contraction. Furthermore, by Hirsch-Pugh-Shub
structural stability for \ph diffeomorphisms, we can perturb these
basic examples in $C^1$ topology to get more examples of \ph skew
products. Examples where $M$ is a non-trivial fiber bundle are also
possible.

\subsection{Statements of results}

We say that a \ph \diff-sm is {\it dynamically coherent} if there
exists a foliation $\Wc$ tangent to $\Ec$ and any curve tangent to
$\Ec$ is contained in a leaf of $\Wc$.

Now we are ready to give the precise statements of our results.

\begin{theorem}\label{TA}
Let $f \colon M\to M$ be a $C^1$ \dc\ph \diff-sm with compact
center foliation $\Wc$. Assume that $\dim \Es=\dim
\Eu=1$. Then volumes of the center leaves are uniformly bounded and
$f$ admits a finite covering \diff-sm $\tilde f\colon \widetilde M
\to \widetilde M$ which is a \ph skew product. Moreover, $\tilde f$
fibers over a hyperbolic automorphism of the 2-torus.
\end{theorem}

\begin{theorem}\label{TB}
Let $f$ be a $C^1$ \dc\ph \diff-sm with compact center foliation
$\Wc$. Assume that the leaves of $\Wc$ are simply
connected. Then volumes of center leaves are uniformly bounded and
$f$ is a \ph skew product.

Assume additionally that $\dim \Eu = 1$, then $f$ fibers over
a hyperbolic automorphism of the torus.
\end{theorem}

The following is our main result.
\begin{TC}
Let $f$ be a $C^1$ \dc\ph \diff-sm with compact center foliation
$\Wc$. Assume that $\dim \Eu \le 2$, $\dim \Es \le 2$
and $\dim \Ec =1$ so that $\Wc$ is a foliation by circles. Then
the lengths of the center leaves are uniformly bounded. Moreover,
every center leaf has a finite holonomy group.
\end{TC}

\begin{question}
 It is not clear to us whether the assumption $\dim \Ec =1$ is crucial for our approach to work.
Can one generalize our techniques to higher dimensional center foliation?
\end{question}

Based on the Main Theorem and a theorem of Bohnet on codimension 4
finite holonomy center foliations~\cite[Theorem 2.64, pp. 116-117]{Boh}, it is easy to establish the following.
\begin{cor}\label{cor_seifert}
Let $f$ be as in the Main Theorem. Then there is a finite cover $\tilde f\colon\widetilde M\to\widetilde M$
with center foliation $\widetilde{\W}^c$ such that $\widetilde {\W}^c$ is a Seifert fibration
on $M$. The holonomy groups of exceptional fibers are products of two cyclic groups.
\end{cor}
\begin{question}
In Corollary~\ref{cor_seifert} passing to a finite cover is needed in order to orient $\Wu$ and $\Ws$.  Is it possible to eliminate exceptional fibers by passing to further finite covers? More generally, can one have
a center foliation which is a Seifert fibration with exceptional fibers such that there is no finite cover in which  
the exceptional fibers disappear?
\end{question}

Examples in~\cite{Boh} show that it is necessary to pass to a finite cover that orients $\Wu$ and $\Ws$, otherwise
one might have infinitely many leaves with non-trivial finite holonomy.

\begin{remark}
\begin{enumerate}
\item[]
\item Passing to a finite cover in Theorem~\ref{TA} is only needed to make
sure that foliations $\Ws$, $\Wc$ and $\Wu$ are orientable. For an
example of a \ph skew product with non-orientable foliations see~\cite{BW}. In this
example the center foliation is a Seifert fibration over the
``pillow-case'' orbi\-fold.
\item The first part of Theorem~\ref{TA}  for 3-dimensional manifolds is a corollary of
results from~\cite{BW}.
\item The center foliation $\Wc$ in Theorem~\ref{TA} is a codimension 2 foliation. One may 
wonder if the volumes of the center leaves are uniformly bounded without making any
dynamical assumptions. As discussed in~\ref{sec_codim_2} this is indeed the case.
However, the results in~\cite{EMS, V} use  $C^1$-smoothness
of the foliation in a crucial way. Therefore we cannot apply them in
our setting since the center foliation is only continuous
transversally. 
\item Note that in Theorem~\ref{TA} we do not assume that the leaves
of the center foliation are all homeomorphic. In fact, in the setup of
Theorem~\ref{TA} the center leaves for $f$ are not necessarily
homeomorphic. (The center leaves for $\tilde f$ are, of course,
homeomorphic.)
\item The first part of Theorem~\ref{TB} is actually immediate from 
Reeb Stability Theorem.
\item Most of our arguments in the proof of the Main Theorem work for stable and unstable foliations of arbitrary
dimension. The assumption on dimension is only used at the end of
Section~\ref{section_boundary} in Proposition~\ref{prop_C}.
\end{enumerate}
\end{remark}

\subsection{Work of Bohnet and Carrasco} In recent theses by Doris Bohnet~\cite{Boh} and Pablo Carrasco~\cite{Car}
the same topic was pursued independently of each other and of our work. There is a certain overlap in results established.
In particular, Theorem~\ref{TA} is also established in both theses. Moreover, it is shown in~\cite{Boh}
that $\widetilde M$ from Theorem~\ref{TA} is actually a two or one fold cover and the center foliation of $f$ is a 
Seifert fibration with either 4 or 0 exceptional fibers whose holonomy groups has two elements $\{id,-id\}$.
A similar description is provided in~\cite{Boh} for codimension 1 partially hyperbolic diffeomorphisms 
whose compact center foliation has finite holonomy. This generalizes our Theorem~\ref{TB} to other foliations 
with finite holonomy (whose leaves are not necessarily simply connected).

\subsection{Acknowledgments}
The author would like to thank Federico Rodriguez Hertz, Andr\'e Henriques and Elmar Vogt for useful communications. The author is very grateful to Pablo Carrasco who pointed out an error in the original proof of the Main Theorem.
Also the author is grateful to Doris Bohnet who kindly sent her thesis to the author.

\settocdepth{subsection}

\section{Preliminaries from foliation theory}\label{section_foliation}
Here we collect well known results on foliations that we will need. We will consider continuous
foliations with $C^1$ leaves. A foliation $\W$ is a {\it continuous foliation with $C^1$ leaves}
if it is given by continuous charts, the leaves $\W(x)$, $x\in M$, are $C^1$ immersed submanifolds and the
tangent space $T_x\W(x)$ depends continuously on $x\in M$.
\subsection{Reeb stability and foliations with simply connected leaves}
Let $\W$ be a continuous foliation with $C^1$ leaves
on a compact manifold $M$. Then given a point $x$ one can define the germinal
holonomy group $G_x(\W)$ that consists of germs of holonomy homeomorphisms
of a small transversal  centered at $x$. One also has a surjective
homomorphism
$$
h^c\colon\pi_1(\W(x),x)\to G_x(\W).
$$
\begin{theorem}[Reeb Stability]
 If a compact leaf $\W(x)$ has trivial holonomy $G_x(\W)$ then
there is a \nbhd  of $\W(x)$ in $M$ that is a union of leaves that are homeomorphic
to $\W(x)$.
\end{theorem}
\begin{cor}\label{cor_Reeb}
 If $\W(x)$ is a compact simply connected leaf then there is a $\W$-saturated \nbhd
  of $\W(x)$ in $M$ that is homeomorphic to $T\times \W(x)$ via a homeomorphism
that takes the leaves of $\W$ to the fibers $\{\cdot\}\times\W(x)$.
\end{cor}
\begin{theorem}[Generalized Reeb Stability]
If a compact leaf $\W(x)$ has a finite holonomy group $G_x(\W)$ then there is an arbitrarily small
foliated normal neighborhood $V$ of $\W(x)$ and a projection $p\colon V\to\W(x)$ such that $(V,\W|_V,p)$
is a foliated bundle with all leaves compact. Furthermore, each leaf $\W(y)\subset V$ has finite
holonomy group of order at most $|G_x(\W)|$ and the covering $p|_\W(y)\colon\W(y)\to\W(x)$ has $k$ sheets, 
where $k\le|G_x(\W)|$.
\end{theorem}

We refer to the book~\cite{CC} for a detailed discussion and the proofs of the above results.

\subsection{The bad set of a compact foliation}
Consider a compact foliation $\W$ on a manifold $M$, the leaf space $X$ and
the volume function $vol\colon X\to (0,+\infty)$ as in~\ref{section_vol}.

Define {\it the bad set} 
$$
\EuScript B\stackrel{\mathrm{def}}{=}\{x\in M: vol \;\;\mbox{is not locally bounded at}\;\W(x)\}.
$$  
Clearly $\EuScript B$ is compact. Also consider the set $\EuScript B'\supset\EuScript B$ --- the set 
of points at which $vol$ is not continuous.
Recall that by Proposition~\ref{prop_hausdorff} $vol$ is lower semi-continuous. 
Any lower semi-continuous 
function is the limit of an increasing sequence of continuous functions and the Baire Category Theorem
implies that $\EuScript B'$ has empty interior. Thus $\EuScript B$ is a compact set with empty interior.

\section{Preliminaries from \ph dynamics}
In this section we collect various preparatory results, which apply
to a wider class of \ph diffeomorphisms than just \ph diffeomorphisms
with compact center foliations.

\subsection{Notation}
Define $E^{cs}=E^c\oplus E^s$ and $E^{cu}=E^c\oplus E^u$.

We write $\Ws$, $\Wu$ and $\Wc$ for the stable, unstable and center foliations, that is,
foliations tangent to $E^s$, $E^u$ and $E^c$ respectively. These are continuous foliations with $C^1$ leaves.
It is 
known that the leaves $\Ws(x), \Wu(x)$, $x\in M$, are diffeomorphic to $\mathbb R^{\dim E^s}$
and $\mathbb R^{\dim E^u}$ respectively. Define the {\it local leaves}
$$
\W^\sigma(x,\e)\stackrel{\mathrm{def}}{=}\{y\in\W^\sigma(x):  d^\sigma(y,x)<\e\},\;\;\; \sigma=s, u, c,
$$ 
where $d^\sigma$ is the metric induced on the leaves of $\W^\sigma$ by the Riemannian metric. If $\e>0$
is small we will sometimes refer to the local leaves as {\it plaques}.
\subsection{The center-stable and center-unstable leaves}
\begin{prop}\label{prop_cs_local}
Let $f$ be a \ph \diff-sm and let $B^c$ be a  ball tangent to
$\Ec$ at every point. Then for any $\varepsilon >0$ the set
$$
\W^\sigma(B^c,\e)\stackrel{\textup{def}}{=}\bigcup_{y\in
B^c}\W^\sigma(y,\varepsilon), \; \sigma = s, u,
$$
is a $C^1$ immersed submanifold tangent to $E^{c\sigma}$, that is,
for every $x\in \W^\sigma(B^c,\e)$  $T_x\W^\sigma(B^c)=E^{c\sigma}$. If
$\varepsilon$ is sufficiently small then $\W^\sigma(B^c,\e)$ is
injectively immersed.
\end{prop}

For a proof see, \eg Proposition 3.4 in~\cite{BBI}. They considered
the low dimensional situation, but the proof extends to higher
dimension straightforwardly.

\begin{cor}\label{cor_cs_epsilon}
Given a complete center leaf $\E$ of a \ph \diff-sm and
$\varepsilon>0$ define
\begin{equation}\label{cu_nbhd}
\W^\sigma(\E, \varepsilon)\stackrel{\mathrm{def}}{=}\bigcup_{y\in
\E}\W^\sigma(y,\varepsilon), \; \sigma = s, u.
\end{equation}
Then $\W^\sigma(\E, \varepsilon)$ is a $C^1$ immersed submanifold
tangent to $E^{c\sigma}$.
\end{cor}

Next we show that $\W^\sigma(\E, \varepsilon)$ is also ``foliated"
by the local center leaves in the sense which is made precise below in Proposition~\ref{prop_local_center_leaves}. 

We say that $E^c$ is {\it weakly integrable} if for every
$ x\in M$ there is an immersed complete $C^1$ manifold $\Wc(x)$ which contains $x$
 and is  tangent to $E^c$ everywhere, that is, $T_y\Wc(x) = E(y)$ for
each $y\in\Wc(x)$.
\begin{prop}\label{prop_local_center_leaves} Assume that $\Ec$ is weakly integrable.
Given a complete center leaf $\E$ of a \ph \diff-sm define
$\W^\sigma(\E, \varepsilon)$ by~\eqref{cu_nbhd}. Then for every
$x\in\W^\sigma(\E, \varepsilon)$ there exists a ball $B^c
\subset\W^\sigma(\E, \varepsilon)$, that contains $x$ and is tangent to $\Ec$ at every point.
\end{prop}

\begin{proof}
Let $\bar\sigma=u$ if $\sigma=s$ and $\bar \sigma=s$ if
$\sigma =u$.

Fix a point $x\in \W^\sigma(\E, \varepsilon)$. By weak integrability
there exists a small ball $\widetilde B^c$ around $x$ that is
 tangent to $\Ec$. Then by
Proposition~\ref{prop_cs_local}, $\W^{\bar \sigma}(\widetilde B^c,
\delta)$, $\delta>0$, is a $C^1$ submanifold tangent to
$E^{\bar\sigma c}$. Clearly $\W^\sigma(\E, \varepsilon)\cap\W^{\bar
\sigma}(\widetilde B^c, \delta)$ is a $C^1$ immersed submanifold
tangent to $\Ec$. To finish the proof we let $B^c$ be the connected
component of $x$ in $\W^\sigma(\E, \varepsilon)\cap\W^{\bar
\sigma}(\widetilde B^c, \delta)$.
\end{proof}

\subsection{The center-stable and center-unstable leaves for a \dc \ph
diffeomorphisms}\label{section_cs_leaves_dc}
Given a center leaf $\E$ define 
\begin{equation}\label{def_cs_global}
\W^\sigma(\E)\stackrel{\mathrm{def}}{=}\bigcup_{\varepsilon>0}\W^\sigma(\E,\varepsilon), \;
\sigma = s, u.
\end{equation}
Clearly $\W^\sigma(\E)$ is a $C^1$ immersed submanifold as well.

 Let us assume now that $f$ is dynamically coherent (the center 
distribution $\Ec$ integrates uniquely to a foliation
$\Wc$). Then Proposition~\ref{prop_local_center_leaves} implies that for
any center leaf $\Wc(x)$ and any $y\in\W^\sigma(\Wc(x))$ the local
center manifold $\Wc(y,\delta)$ is contained in $\W^\sigma(\Wc(x))$
for sufficiently small $\delta$. Note that it does not follow that
the whole leaf $\Wc(y)$ is in $\W^\sigma(\Wc(x))$. The intersection
$\Wc(y)\cap\W^\sigma(\Wc(x))$ is some open subset of $\Wc(y)$.

\subsection{The accessible boundary}\label{section_accessible_boundary}
Consider the collection $\mathcal A^\sigma$ of smooth curves
$\alpha\colon[0,1]\to M$ such that 
$$
\dot{\alpha}\in \Ec,\;
\alpha([0,1))\subset\W^\sigma(\E),\; \alpha(1)\notin\W^\sigma(\E), \; \sigma=s,u.
$$
Equip $\mathcal A^\sigma$ with the following equivalence relation. Two curves
$\alpha_0$ and $\alpha_1$ are equivalent if there is a continuous homotopy
$\alpha_t, t\in[0,1]$ connecting $\alpha_0$ and $\alpha_1$ such that for all $t\in [0,1]$ $
\alpha_t\in \mathcal A^\sigma$  and $\alpha_t(1)=\alpha_0(1)$.

Define the {\it accessible boundary} of $\W^\sigma(\E)$ as the set of equivalence classes
 $$
\label{def_boundary}
\partial \W^\sigma(\E)\stackrel{\mathrm{def}}{=}\{[\alpha]: \alpha\in \mathcal A^\sigma\}, \; \sigma=s,u.
$$
Also define the {\it closure} $\W^\sigma(\E)^{\mathrm{cl}}\stackrel{\mathrm{def}}{=}\W^\sigma(\E)\sqcup\partial\W^\sigma(\E)$.
We equip the closure with the obvious topology. Note also that there is a natural projection
$\pi\colon\partial\W^\sigma(\E)\to M$, $[\alpha]\to\alpha(1)$. However $\partial \W^\sigma(\E)$
cannot be identified with a subset of $M$ since in general $\pi$ is not injective.

The following proposition is easy to prove. (See, \eg Proposition
1.13 in~\cite{BW} for a proof in dimension 3.)

\begin{prop}\label{prop saturated}
The boundary $\partial\W^\sigma(\E)$ is saturated by the leaves of
$\W^\sigma$, $\sigma=s, u$. This means that for every $[\alpha]\in\partial \W^\sigma(\E)$
there exists a continuous map $\W^\sigma(\alpha(1))\to\partial \W^\sigma(\E)$ that fits into
the commutative diagram 
$$
\xymatrix{
\W^\sigma(\alpha(1)) \ar@{->}[rd]_i \ar@{->}[r]
& \partial \W^\sigma(\E)\ar@{->}[d]^\pi\\
&M
 }
$$
where $i$ is the natural immersion of $\W^\sigma(\alpha(1))$, $\sigma=s, u$.
\end{prop}

\begin{cor}\label{cor_codim_1_boundary}
 If $\dim E^c=1$ then $\W^\sigma(\E)^\mathrm{cl}$ is a manifold with smooth boundary 
$\partial \W^\sigma(\E)$ which is a union of leaves of $\W^\sigma$ (some of which can
be repeated), $\sigma=s, u$.
\end{cor}

\begin{figure}[htbp]
\begin{center}
\includegraphics{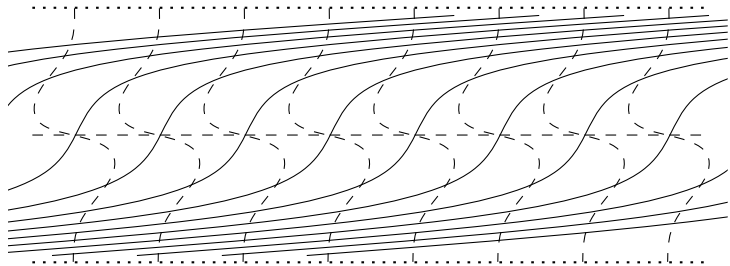}
\end{center}
 \caption{The solid curves represent the stable leaves and the dashed curves represent the center leaves.
The horizontal dashed line $\E$ is tangent to the center distribution at every point. However
$\E$ is not a leaf of the invariant center foliation. Still one can form $\Ws(\E)$ which is
an open strip with two boundary components drawn by dotted lines. Note that on the manifold
these two boundary components are represented by the same stable leaf.  }
\label{fig1}
\end{figure}

\begin{rmk}
To the best of our knowledge currently there is no example of a \dc \ph
\diff-sm of a compact manifold with a center leaf $\E$ (not necessarily compact) such
that $\partial \Ws(\E)\ne \varnothing$. However, an example
of a partially hyperbolic diffeomorphism with a one dimensional center distribution
that does not integrate uniquely was constructed~\cite{RHDC}. In this example
there is a curve $\E$ tangent to $E^c$ and diffeomorphic to $\mathbb R$ such that $\Ws(\E)$ is an open strip
with two boundary components as shown on Figure~\ref{fig1}.
\end{rmk}

\subsection{Volume recurrent center leaves}

A compact center leaf $\E$ is called {\it volume recurrent} if 
$$
\liminf_{n\to\infty} vol(f^n\E)<+\infty.
$$

The following proposition will not be used in the proofs but we include it since it allows a
better description of the accessible boundary of the lifted center-stable manifolds $\partial\widehat\W^s(\widehat\E)$, see Remark~\ref{rmk_acc_boundary}.
\begin{prop} \label{prop_finite_intrscn} Let $f$ be a \dc \ph \diff-sm 
with a compact volume recurrent center leaf $\E$. Then for any $x\in\E$
the intersection $\W^\sigma(x)\cap\E$ is a finite set. Moreover, the number
of points of intersection, $j_\sigma(x)\stackrel{\mathrm{def}}{=}\#(\W^\sigma(x)\cap\E)$,  is an upper semi-continuous function on $\E$, $\sigma=s, u$.
 \end{prop}
\begin{proof}
 Choose a sufficiently small $\e>0$ so that the following property holds: for any $y\in M$
and for any $z\in \Ws(y,\e)$, $z\neq y$, the local center leaves $\Wc(y,\e)$ and
$\Wc(z,\e)$ are disjoint. Define
$$
\mu\stackrel{\mathrm{def}}{=}\inf_{y\in M} vol(\Wc(y,\e)).
$$
Clearly $\mu>0$.

Consider a finite set $\{x_1, x_2,\ldots x_k\}\subset\Ws(x)\cap\E$. Then there exists a sufficiently large $N$
such that $\forall n\ge N$ the set $\{f^nx_1, f^nx_2,\ldots f^nx_k\}$ will be contained 
in a local stable manifold of size $\e$ and hence the local center manifolds $\Wc(f^nx_1,\e), 
\Wc(f^nx_2,\e),\ldots \Wc(f^nx_k,\e)$
will be disjoint. Note that these local center manifolds are subsets of the leaf $f^n\E$.
It follows that
$$
\forall n\ge N  \; vol(f^n\E)\ge k\mu.
$$
This yields an upper bound on $k$
$$
k\le \frac1\mu\liminf_{n\to\infty} vol(f^n\E).
$$

Hence $\Ws(x)\cap\E$ is a finite set. Finally notice that continuity of $\Ws$ inside $\Ws(\E)$
implies that for any $x\in\E$ and any $y\in\E$ which is sufficiently close to $x$ 
$\#(\Ws(y)\cap\E)\ge\#(\Ws(x)\cap\E)$. Therefore function $j_s$ is upper semi-continuous.

By reversing the time we also see that $\#(\Wu(x)\cap\E)$ is bounded and upper semi-continuous.
\end{proof}

\subsection{The splitting of the center holonomy}\label{section_splitting}
The following statement is an immediate corollary of the uniqueness of the center leaves.
\begin{prop}\label{prop_product_holonomy}
 Assume that $f$ is a \dc \ph \diff-sm. Then the holonomy group of the center foliation
splits as a product $G_x(\Wc)=G_x^s(\Wc)\times G_x^u(\Wc)$, where $G_x^s(\Wc)$ and $G_x^u(\Wc)$
are the holonomy groups of $\Wc$ inside $\Ws(\Wc(x))$ and $\Wu(\Wc(x))$ respectively.
\end{prop}

\section{Anosov homeomorphisms}

A \ph skew product $f\colon M\to M$ projects to a \homeo $\bar
f\colon X\to X$. Clearly $\bar f$ must be ``uniformly hyperbolic" in
some sense. However, there is no smooth structure on $X$ which is
compatible with dynamics. (In fact, it is not clear how to equip $X$ with some smooth structure.) This motivates a more general definition of an Anosov
homeomorphism.

\subsection{Definition}
Let $(X,\rho)$ be a compact metric space. For a \homeo $h\colon X\to
X$ define
\begin{multline*}
\qquad\Ws(x,\varepsilon)\stackrel{\mathrm{def}}{=}\{y\in X\colon \rho(h^nx, h^ny)\leq\varepsilon, \forall  n\geq 0\}, \\
\qquad\shoveleft{\Wu(x,\varepsilon)\stackrel{\mathrm{def}}{=}\{y\in X\colon \rho(h^nx,
h^ny)\leq\varepsilon, \forall  n\leq 0\}.\hfill}
\end{multline*}
We say that $h\colon X\to X$ is an {\it Anosov \homeo} if there
exist  $\varepsilon>0$, $\delta>0$, and constants $C>0$,
$\lambda \in (0,1)$ such that
\begin{multline*}
\mbox{(AH1)}\qquad y\in \Ws(x, \varepsilon) \Rightarrow \rho(h^nx,
h^ny)
\leq C\lambda^n\rho(x,y), n\geq 0, \\
\qquad \quad\;\;\,\quad\shoveleft{y\in \Wu(x, \varepsilon)
\Rightarrow
\rho(h^{-n}x, h^{-n}y) \leq C\lambda^n\rho(x,y), n\geq 0. \hfill}\\
\end{multline*}
\quad \,(AH2)\qquad If $\rho(x,y)\leq\delta$ then there is a unique
point of intersection
$$
[x,y]\stackrel{\mathrm{def}}{=}\Ws(x,\varepsilon)\cap\Wu(y,\varepsilon)
$$
\qquad \quad\qquad\quad and the map $[\cdot,\cdot]\colon \{(x, y)\in
X^2 \colon \rho(x,y)\leq\delta\}\to X$ is continuous.

\begin{rmk}
The original definition of Anosov \homeo was given by Bowen \cite{B} and by Alekseev and
Yakobson~\cite{AY}  in order to axiomatize the
conditions needed for the construction of Markov partitions. Alekseev
and Yakobson use term ``$A^\#$-homeomorphism.'' Their definition is
identical to the above one except that they require constant $C$ to
be equal to 1. Soon we will see that for our purposes the above
definition is more convenient. The difference is actually very minor.
The usual adapted metric construction is applicable in this
setting: given an Anosov \homeo $h\colon (X, \rho)\to(X,\rho)$ one
can construct an equivalent metric $\tilde \rho$ such that $h\colon (X,
\tilde\rho)\to(X,\tilde\rho)$ is $A^\#$-homeomorphism in the sense
of Alekseev and Yakobson.
\end{rmk}

Given an Anosov \homeo $h$ define global stable and unstable sets in
the standard way:
\begin{multline*}
\qquad\Ws(x)\stackrel{\mathrm{def}}{=}\{y\colon \rho(h^nx, h^ny)\to 0, \;\; n\to +\infty\}, \\
\qquad\shoveleft{\Wu(x)\stackrel{\mathrm{def}}{=}\{y\colon \rho(h^{-n}x, h^{-n}y)\to 0,\;\;
 n\to +\infty\}.\hfill}
\end{multline*}

Now let $X$ be a topological manifold. We say that $h\colon X\to X$
is {\it Anosov \homeo with stable and unstable foliations} if $h$ is
an Anosov \homeo and the stable and unstable sets form two continuous
topologically transverse foliations of $X$.  In this case it makes sense to speak about
dimensions of $\Ws$ and $\Wu$.

\subsection{Partially hyperbolic skew products and Anosov
homeomorphisms} Let us go back to the fibration $p\colon M\to X$,
the \ph skew product $f\colon M\to M$ and the factor $\bar f \colon
X\to X$. 

Endow $X$ with the Hausdorff metric defined as
$$
\rho(a,b)\stackrel{\mathrm{\textup{def}}}{=}\max\left(\max_{x\in \Wc(a)} \min_{y\in
\Wc(b)}d(x, y), \max_{x\in \Wc(b)}\min_{y\in \Wc(a)}d(x,y)\right),
$$
where $d$ is the metric on $M$ induced by the \R metric. Here we slightly abuse the notation by writing $\Wc(\cdot)$ for
$p^{-1}(\cdot)$.

\begin{prop} \label{phsp-Anhaueo}
Homeomorphism $\bar f\colon (X,\rho)\to(X,\rho)$ is an Anosov \homeo
with stable and unstable foliations.
\end{prop}

\begin{proof}
Notice that the local stable and unstable manifolds of $a \in X$ are
simply the projections of $\Ws(\Wc(a),\varepsilon)$ and
$\Wu(\Wc(a),\varepsilon)$ on $X$. It immediately  follows that the stable
and unstable sets of $\bar f$ are, in fact, transverse continuous
foliations. 

Property (AH2) can also be established rather easily. Consider two points $a,
b\in X$ with $\rho(a, b)<\delta$. Then transversality implies that
for an appropriate value of $\varepsilon$ (more precisely, $\varepsilon
= C\delta$, where $C$ depends on the lower bounds of the angles between
$\Es$, $\Ec$ and $\Eu$), $\Ws(\Wc(a),\varepsilon)$ and
$\Wu(\Wc(a),\varepsilon)$ intersect at a single center leaf. This
leaf depends continuously on $(a, b)$ and, hence, (AH2) follows.

Verification of (AH1) requires some work. Let $\A$ be the collection
of piecewise smooth paths $\alpha \colon [0,1]\to M$ such that
$\dot{\alpha}$ lies in one of the distributions $\Es$, $\Ec$ or
$\Eu$ whenever $\alpha$ is differentiable. Define a new metric,
$d^{scu}$, on $M$ in the following way
$$
d^{scu}(x,y)=\inf_{\alpha\in\A, \alpha(0)=x,
\alpha(1)=y}(length(\alpha)).
$$
Define the corresponding Hausdorff metric on $X$:
$$
\rho^{scu}(a,b)\stackrel{\mathrm{\textup{def}}}{=}\max\left(\max_{x\in \Wc(a)} \min_{y\in
\Wc(b)}d^{scu}(x, y), \max_{x\in \Wc(b)}\min_{y\in
\Wc(a)}d^{scu}(x,y)\right).
$$
Since $d^{scu}$ is equivalent to $d$ we also have that $\rho$ is
equivalent to $\rho^{scu}$. Hence it is clear that \homeo $\bar f$
being Anosov with respect to $\rho^{scu}$ implies $\bar f$ being Anosov with respect to $\rho$ ---
one only needs to change the values of $C$, $\varepsilon$ and
$\delta$ from the definition of Anosov homeomorphism. Therefore, it
is sufficient to show that $\bar f\colon(X, \rho^{scu})\to(X,
\rho^{scu})$ satisfies (AH1).

The proof of the first inequality (the second one is analogous) from
(AH1) is based on the following observation:

1. Fix a point $b\in M$.

2. For a sufficiently small $\varepsilon>0$ and a point $a\in \Ws(\Wc(b), \varepsilon)$ one can choose a sequence of
paths $\{\alpha_n\}$ from $\A$ that connect $a$ and $b$ such that
$$\underset{n\to \infty}{\lim}length(\alpha_n)=d^{scu}(a,b)$$ and the paths $\alpha_n$ do not have components that lie in $\Wu$.

3. Given any $\kappa>0$ one can find a suffiently small $\varepsilon>0$ such that,
if $a\in \Ws(\Wc(b), \varepsilon)$ realizes the minimum $\min_{x\in\Wc(a)}d^{scu}(x,b)$  and
$\{\alpha_n\}$ is a sequence of paths connecting $a$ and $b$ as above, then the lower limit of the ratio of the total
length of components of $\alpha_n$ that lie in $\Ws$ to the length
of $\alpha_n$ is greater than $1-\kappa$.

We omit the estimate itself as it is rather standard.
\end{proof}
\subsection{The classification of codimension 1 Anosov homeomorphisms with stable and
unstable foliations}\label{subs_codim_1}
Recall that given a \homeo $h\colon X\to X$, a point $x\in X$ is
called a {\it wandering point} if there exists an open neighborhood
$\U$ of $x$ such that the sets $h^n(\U), n \geq 0$, are mutually
disjoint. Otherwise $x$ is said to be a {\it non-wandering point}.
The set of non-wandering points of $h$ will be denoted by $NW(h)$.

The following are generalizations of the theorems of Franks~\cite{F} and Newhouse~\cite{N} to the
setting of Anosov homeomorphisms.
\begin{theorem}\label{codim_1_F}
Let $h\colon X\to X$ be a codimension 1 Anosov \homeo  with stable and unstable foliations. 
Assume that
 $NW(h)=X$ then $X$ is homeomorphic to a torus and $h$
conjugate to an Anosov automorphism of the torus.
\end{theorem}
\begin{theorem}\label{codim_1_N}
 If \homeo $h\colon X\to X$ is a codimension 1 Anosov \homeo  with stable and unstable foliations
then $NW(h)=X$.
\end{theorem}
Short proofs of the Franks-Newhouse Theorems were given by Hiraide~\cite{H}. These proofs 
transfer rather easily to the topological setting.

The proof of the Newhouse Theorem starts with the spectral decomposition for the Anosov diffeomorphism
and then proceeds with a soft topological argument that does not require any alternation in the topological setting. The spectral
decomposition in our setting comes from the spectral decomposition for a subshift of finite type since
Anosov homeomorphisms admit Markov partitions (see, \eg~\cite{AY}).

The proof of the Franks Theorem starts with establishing global product structure of the stable and unstable
foliations on the universal cover. Assuming that the unstable foliation is one dimensional,
a ``nice'' measure class $\mu^u$ is constructed on the leaves of the unstable foliation which is invariant 
under the holonomy along the stable foliation. This step works well in the topological setting.
The remaining arguments in~\cite{H} are devoted to the construction of the group
of deck transformations. With only minor adjustments they apply well for Anosov homeomorphisms with
stable and unstable foliations with only minor adjustments.

\section{The proof of Theorem~\ref{TA}}

Let $\widetilde M$ be a finite cover of $M$ such that the stable,  
center and unstable foliations lift to orientable foliations $\widetilde\W^s$,
$\widetilde\W^c$ and $\widetilde\W^u$. Then $f$ lifts to a partially 
hyperbolic diffeomorphism $\tilde f\colon\widetilde M\to\widetilde M$.
We will show that $\tilde f$ is the \phsp posited in Theorem~\ref{TA}. Since $f$ and $\tilde f$
are \dc the discussion in Sections~\ref{section_splitting} and~\ref{section_cs_leaves_dc}
applies and we can speak about the stable and unstable holonomy.
\begin{lemma}\label{lemma_codim_1}
 For any $x\in\widetilde M$ the holonomy groups $G_x^\sigma(\widetilde\W^c)$, $\sigma=s, u$, are
trivial.
\end{lemma}
The lemma and Proposition~\ref{prop_product_holonomy} imply that the full holonomy
groups $G_x(\widetilde\W^c)$ are trivial. Then the Reeb Stability Theorem implies
that $\tilde f$ is a \ph skew product and the leaf space of $\widetilde\W^c$ is a 2-dimensional topological 
manifold
$X$. By Proposition~\ref{phsp-Anhaueo} diffeomorphism $\tilde f$ projects to an Anosov 
homeomorphism $h\colon X\to X$. And by Theorems~\ref{codim_1_F} and~\ref{codim_1_N} \homeo $h$ is conjugate
to a hyperbolic automorphism of $\T^2$.

Hence to finish the proof of Theorem~\ref{TA} we only need to prove the above lemma.
\begin{proof}
Assume that $G_x^\sigma(\widetilde\W^c)$ is not trivial.
Choose $\alpha\in\pi_1(\widetilde\W^c(x),x)$ such that $h^c(\alpha)$ is non-trivial in $G_x^\sigma(\widetilde\W^c)$.
 For a sufficiently small $r>0$ the corresponding center holonomy 
$$
H^c(\alpha)\colon\widetilde\W^\sigma(x,r)\to\widetilde\W^\sigma(x)
$$
that represents $h^c(\alpha)$ is well defined. 
Map $H^c(\alpha)$ is an orientation preserving strictly increasing homeomorphism onto its image. 
Also $H^c(\alpha)$ is not identity. This implies that 
either forward or backward $H^c(\alpha)$-orbit of a point consists of distinct points and is contained in $\widetilde \W^\sigma(x,r)$. Denote this orbit
by $\{x_n; n\ge 1\}\subset\widetilde \W^\sigma(x,r)$. Obviously $x_n\in\widetilde\W^c(x_1)$. Also it is clear
that for sufficiently small $\e>0$ the center plaques $\widetilde\W^c(x_n,\e)$, $n\ge 1$, are disjoint. This
gives a contradiction since $vol(\widetilde\W^c(x_1))$ is finite.
\end{proof}

\section{Proof of Theorem~\ref{TB}}

By Corollary~\ref{cor_Reeb} foliation $\Wc$ is a locally trivial fibration over the leaf space $X$. Hence $f$ is a \ph
skew product over $h\colon X\to X$. Homeomorphism $h$ is Anosov by Proposition~\ref{phsp-Anhaueo}. Finally
since $\dim E^u=1$ Theorems~\ref{codim_1_F} and~\ref{codim_1_N} apply: $h$ is conjugate
to a hyperbolic automorphism of a torus.

\section{Accessible boundary is empty} \label{section_boundary}
 Consider a complete center leaf $\E$
of a \ph \diff-sm and the union $\Ws(\E)$ of the stable leaves that pass
through $\E$~\eqref{def_cs_global}. If $\Ws(\E)$ is complete then Proposition~\ref{prop_local_center_leaves}
implies that every point $y\in\Ws(\E)$ is a
center of a center-stable plaque of a
fixed size which is independent of $y$. If $\Ws(\E)$ is not complete
then it has a non-empty accessible boundary $\partial \Ws(\E)$
defined by~\eqref{def_boundary}. Proposition~\ref{prop_local_center_leaves}
still applies, but the size of the center-stable plaque decreases as the center of the 
plaque $y$ approaches $\partial
\Ws(\E)$.

{\bfseries Standing assumption:} {\it Partially hyperbolic diffeomorphism $f\colon M\to M$
is dynamically coherent with oriented one-dimensional compact center foliation $\Wc$. Center leaf $\E$
is length recurrent.}

\subsection{A reduction} \label{section_reduction} Recall that by 
Corollary~\ref{cor_codim_1_boundary} the closure $\Ws(\E)^\mathrm{cl}$ is a manifold
whose boundary is a union of stable leaves. 

For any $x\in\Ws(\E)$ the set $\E\cap\Ws(x)$ is either finite or countable ({cf. Remark~\ref{rmk_acc_boundary}). Each intersection point
in $\E\cap\Ws(x)$ depends continuously on $x\in\Ws(\E)$. Therefore the set
$$
\widehat\W^s(\widehat\E)\stackrel{\mathrm{\textup{def}}}{=}\{(x,q), x\in\Ws(\E), q\in\E\cap\Ws(x)\}
$$
inherits the topology, the smooth structure and the \R metric from $\Ws(\E)$. 

The map $p\colon\widehat\W^s(\widehat\E)\to\Ws(\E)$
$$
p((x,q))\stackrel{\mathrm{\textup{def}}}{=}x
$$
is a local isometry. Therefore the stable and the center foliations lift to foliations
$\widehat\W^s$ and $\widehat\W^c$ in $\widehat\W^s(\widehat\E)$. Each stable leaf
$\Ws(x), x\in \E$, lifts to $\#(\Ws(x)\cap\E)$ stable leaves
$\widehat\W^s((x,q))=(\Ws(x),q)$, $q\in\Ws(x)\cap\E$. It follows that $\widehat\W^s(\widehat\E)$
is saturated by complete stable leaves and can be viewed as the disjoint union
$$
\widehat\W^s(\widehat\E)=\bigsqcup_{x\in\widehat\E}\widehat\W^s(x),
$$
where
$$
\widehat\E\stackrel{\mathrm{def}}{=}\{(x,x), x\in\E\}.
$$
\begin{rmk}
Note that even though the center leaves lift locally to $\widehat\W^s(\widehat\E)$ they do not
necessarily lift globally as $p$ might fail to be a covering map.
\end{rmk}

Now we are ready to define the accessible boundary of $\widehat\W^s(\widehat\E)$. For each $x\in \widehat\W^s(\widehat\E)$ the center leaf $\widehat\W^c(x)$ is either oriented circle or an open segment
$\alpha_x\colon(0,1)\to\widehat\W^s(\widehat\E)$. In the latter case we complete $\alpha_x$ by adding two points $\alpha_x(0)$
and $\alpha_x(1)$ and define
$$
\partial\widehat\W^s(\widehat\E)\stackrel{\mathrm{\textup{def}}}{=}\bigcup_{\widehat\W^c(x)\sim(0,1)}\{\alpha_x(0),\alpha_x(1)\}.
$$
It is clear that if $p$ is one-to-one then the above definition coincides with the one in Section~\ref{section_accessible_boundary}.
We also define $\widehat\W^s(\widehat\E)^\mathrm{cl}\stackrel{\mathrm{\textup{def}}}{=}\widehat\W^s(\widehat\E)\sqcup\partial\widehat\W^s(\widehat\E)$
and metrize it in the following way.  Let $\hat d$ be the metric induced by the \R metric on $\widehat\W^s(\widehat\E)$, that is
$
\hat d(x,y)\stackrel{\mathrm{\textup{def}}}{=}\inf(length(\gamma))
$, where the infimum is taken over all smooth curves that connect $x$ and $y$. If $x, y\in\widehat\W^s(\widehat\E)^\mathrm{cl}$ then let $\alpha_x$ and $\alpha_y$ be corresponding parametrizations
of the center leaves with $\alpha_x(t_x)=x$ and $\alpha_y(t_y)=y$. Extend $\hat d$ in the following way
$$
\hat d(x,y)\stackrel{\mathrm{\textup{def}}}{=}\lim_{t_1\to t_x, t_2\to t_y}\hat d(\alpha_x(t_1),\alpha_y(t_2)).
$$

Now it is a routine exercise to check that $(\widehat\W^s(\widehat\E)^\mathrm{cl},\hat d)$ is the usual
completion of $(\widehat\W^s(\widehat\E),\hat d)$. Another routine exercise is to check that
Corollary~\ref{cor_codim_1_boundary} holds in this more general setting; that is, $\widehat\W^s(\widehat\E)^\mathrm{cl}$
is a manifold with boundary where boundary component of $\alpha_x(\sigma)$ is identified with
$
\Ws\big(\lim_{t\to\sigma} p\circ\alpha (t)\big), \;\;\sigma=0,1.
$

In the same way, for $n\ge 1$, define
$\widehat\W^s(f^n\widehat\E)$, $\partial\widehat\W^s(f^n\widehat\E)$, $\widehat\W^s(f^n\widehat\E)^\mathrm{cl}$ and $f^n\widehat\E$. Finally define
$$
\hat f\colon {\widehat\W^s(f^{n-1}\widehat\E)}\to{\widehat\W^s(f^n\widehat\E)}, \;\;
\hat f\colon (x,q)\mapsto(fx,fq), q\in f^{n-1}\E\cap\Ws(x).
$$
Note that both $f$ and $\hat f$ extend to the closures by continuity. We have the 
following commutative diagram:
$$
\begin{CD}
 \widehat\W^s(\widehat\E)^\mathrm{cl} @>\hat f>>\widehat\W^s(f\widehat\E)^\mathrm{cl}@>\hat f>>\ldots\\
@VVV@VVV@.\\
\Ws(\E)^\mathrm{cl} @>f>>\Ws(f\E)^\mathrm{cl}@>f>>\ldots
\end{CD}
$$
Notice that, a priori, ``$\hat f$ dynamics'' does not come from an ambient \diff-sm. However, 
if one studies ``$\hat f$ dynamics'' in the forward orbit of $\widehat \W^s(\widehat\E)^\mathrm{cl}$ then the 
tools that come from partial hyperbolicity (the center and stable foliations) are available. 

The advantage of considering $\hat f$ instead of $f$ is that for every 
$z\in\widehat \W^s(f^n\widehat\E)$ the intersection $\widehat\W^s(z)\cap f^n\widehat\E$ is a single point. 

\begin{rmk}\label{rmk_acc_boundary} The above reduction is very general and works for center leaves of higher 
dimension which do not have to be volume recurrent. However if $\E$ is one dimensional and length recurrent
then Proposition~\ref{prop_finite_intrscn} allows to say more about the structure of $\partial\widehat\W^s(\widehat\E)$. Namely, components of $\partial\widehat\W^s(\widehat\E)$ can be classified into two classes: 
\begin{enumerate}
\item The lifts of components of $\partial\Ws(\E)$.
\item The lifts of stable leaves $\Ws(y)$ where $y\in \E$ are boundary points of components of the open sets 
$\{x\in \E: j_s(x)\ge l\}$, $l=\overline{1,k}$ where  $j_s$ and $k$ are as in Proposition~\ref{prop_finite_intrscn}.
\end{enumerate}
\end{rmk}

{\bfseries A notational convention:} {\it To avoid heavy notation we take off the hats.
In other words, we will write $\Ws(\E)$ for $\widehat\W^s(\widehat\E)$, $f$ for $\hat f$, and so on.}

\subsection{The holonomy projection}
Pick a point $x\in f^d\E$, $d\ge 0$, and a point $y\in\Ws(x)$. Recall that $\Wc(y)\cap\Ws(f^d\E)$ is an open
subset of $\Wc(y)$. Let $\Dc(y)$ be the connected component of $y$
in $\Wc(y)\cap\Ws(f^d\E)$. In other words $\Dc(y)$ is the
leaf of $\Wc|_{\Ws(f^d\E)}$. If the set $\Dc(y)$ is
distinct from $\Wc(y)$ then $\Dc(y)$ is homeomorphic to an open interval. Also note that $\Dc(fy)=f(\Dc(y))$.

Define the holonomy projection $\Hyx \colon \Dc(y)\to f^d\E$ by taking
the intersection
$$
\Hyx(z)\stackrel{\mathrm{def}}{=}\Ws(z)\cap f^d\E.
$$
This is well defined according to the discussion in Section~\ref{section_reduction}.

\begin{lemma}\label{lemma_finite_covering}
If $\Dc(y)=\Wc(y)$ then $\Hyx$ is a covering map with a finite number
of sheets.
\end{lemma}

\begin{proof}
Clearly the holonomy projection commutes with dynamics: $$f\circ \Hyx
= \H_{fy\to fx} \circ f.$$ Therefore we only need to show that
there exists a positive number $n$ such that $\H_{f^ny\to f^nx}\colon \Wc(f^ny)\to f^{n+d}\E$
is a finite covering map.

Since $\Wc(y)\subset\Ws(f^d\E)$ we have that $\Wc(y)\subset\Ws(f^d\E, R)$
for a sufficiently large $R$. Choose a small $\varepsilon >0$ so that
$\Ws(f^{n+d}\E,\varepsilon)$, $n\ge 0$, ia a fiber bundle over $f^{n+d}\E$ with fibers being
small stable plaques. There exists $n$ such that $f^n(\Wc(y))\subset
f^n(\Ws(f^d\E,R))\subset\Ws(f^{n+d}\E,\varepsilon)$. The leaf
$\Wc(f^ny)=f^n(\Wc(y))$ is an embedded compact manifold transverse
to the fibers of the bundle $\Ws(f^{n+d}\E,\varepsilon)$. It follows that $\H_{f^ny\to f^nx}$ is
a finite covering map.
\end{proof}

We use notation $(z_1, z_2)^c$ to denote a positively oriented open
interval inside of a center leaf with endpoints $z_1$ and $z_2$. And we
write $\overline{(z_1, z_2)^c}$ for a negatively oriented center
interval.

\begin{lemma}\label{lemma_countable_covering}\footnote{The proof of this Lemma~\ref{lemma_countable_covering}
is the only place where the length recurrence assumption is used.}
If $\Dc(y_0)\ne \Wc(y_0)$, $y_0\in\Ws(\E)$, then $\H_{y_0\to x}\colon \Dc(y_0)\to f^d\E$ is a covering map with countable number
of sheets.
\end{lemma}

\begin{proof}
Let $\tau\colon S^1\to\E$ be the positively oriented generator of
$\pi _1(\E)=\Z$. 
As in the proof of the previous lemma choose $\e>0$ sufficiently small
so that $\Ws(f^n\E,\e)$, $n\ge d$, is a fiber bundle with small stable plaques as fibers.

Since $\E$ is length recurrent we can choose an increasing sequence of times
$\{k_i;i\ge1\}$ such that
\begin{equation}\label{eq_subs}
\forall i\ge1 \;\; length(f^{k_i}\E)<C
\end{equation}
for some $C>0$.
\begin{claim}
 For any $m\ge1$ there exists $\delta>0$ such that for any $i\ge1$ and any
$z_0\in\Ws(f^{k_i}\E,\delta)$ the loops $f^{k_i}\circ\tau^m$ and
$f^{k_i}\circ\tau^{-m}$ lift along the stable plaques to $(z_0,z_m)^c$ and
$\overline{(z_0, z_{-m})^c}$.
\end{claim}
This claim is a direct corollary of~\eqref{eq_subs} and uniformity of the local product structure.

Choose a sufficiently large $i$ so that $z_0\stackrel{\mathrm{def}}{=}f^{k_i}y_0\in\Ws(f^{k_i}\E,\delta)$.
Then, by the above claim, the loops $f^{k_i}\circ\tau^m$ and
$f^{k_i}\circ\tau^{-m}$ lift to $(z_0,z_m)^c$ and
$\overline{(z_0, z_{-m})^c}$. By pulling back with $f^{-k_i}$ we get that $\tau^m$ and $\tau^{-m}$
lift to $(y_0,y_m)^c$ and
$\overline{(y_0, y_{-m})^c}$. This was sequence $\{y_m; m\in\Z\}$ is defined. It is obvious that $\{y_m; m\in\Z\}$ is an increasing sequence in $\Dc(y_0)$.

Let $a=\underset{m\to\infty}{\lim}y_{-m}$ and $\varkappa(a)
=\underset{m\to\infty}{\lim}y_m$. It is clear from our construction that $\H_{y_0\to x}\colon(a,\varkappa(a))^c\to\E$
is a covering map and that $\H^{-1}_{y_0\to x}(x)\supset\{y_m; m\in\Z\}$.

Assume for a moment that $a\in\Dc(y_0)$. Then there exists an $n>0$
such that $f^na$ belongs to $\Ws(f^n\E, \varepsilon)$.  It follows that $f^na$ has a
small neighborhood in $\Dc(f^na)$ that contains no more than one point (one when 
$a\in \H^{-1}_{y_0\to x}(x)$) from
$\H_{f^ny_0\to f^nx}^{-1}(f^nx)$. Consequently, $a$ is not an
accumulation point of $\H^{-1}_{y_0\to x}(x)$ in $\Dc(y_0)$, which
contradicts the definition of $a$.

Thus we see that $a$ is a boundary point of $\Dc(y_0)$. Analogously
$\varkappa(a)$ is the other boundary point of $\Dc(y_0)$.

We conclude that $\Dc(y_0)=(a, \varkappa(a))^c$ and $\H_{y_0\to
x}\colon\Dc(y_0)\to\E$ is a covering map.
\end{proof}

\begin{rmk}
The sequence $\{y_m, m\in\Z\}$ and the endpoints $a, \varkappa(a)$ 
constructed in the proof above will
be used throughout the current section.
\end{rmk}

\subsection{The center holonomy}
Consider the flow generated by the positively oriented vector field
tangent to the center foliation $\Wc$. Here we will argue that the
first return map to $\Ws(y_0)$ of this flow is a well defined
homeomorphism $\H_{y_0}^c\colon\Ws(y_0)\to\Ws(y_0)$.

It is clear from Lemma~\ref{lemma_countable_covering} that if we start
at $y_0$ and follow the leaf $\Wc(y_0)$ in positive direction then we will
return to $\Ws(y_0)$ at $y_1$ and this is the first return, \ie
$(y_0,y_1)^c\cap\Ws(y_0)=\varnothing$. Similarly for any
$z_0\in\Ws(y_0)$ with $\Dc(z_0)\ne \Wc(z_0)$ the simple argument in
the proof of Lemma~\ref{lemma_countable_covering} can be repeated and
yields a well defined point of first return $z_1\in \Ws(y_0)$. For
any such $z_0$ define $\H^c_{y_0}(z_0)=z_1$.

If $\Dc(z_0)=\Wc(z_0)$ for a point $z_0\in \Ws(y_0)$ then  we have that
the holonomy projection $\H_{z_0\to\Ws(z_0)\cap\E}\colon\Wc(z_0)\to\E$ is a finite covering
by Lemma~\ref{lemma_finite_covering}. Thus the first return time is
well defined as well.

Finally, we remark that continuity of $\Wc$ implies that
$\H^c_{y_0}$ is a homeomorphism.

\subsection{Structure of the center holonomy}
Here we reveal some topological structure on $\Ws(y_0)$ which is
respected by the center holonomy $\H^c_{y_0}$.

Recall that by Proposition~\ref{prop saturated} (and the discussion in Section~\ref{section_reduction})
$\Ws(a)\sqcup\Ws(\varkappa(a))\subset \partial\Ws(\E)$. 
For any point $p\in \Ws(a)$ a sufficiently small center interval $(p, p')^c$ is
contained in $\Ws(\E)$. On the other hand, $\Wc(p)\ne \Dc(p')$ and
hence there exists some $\varkappa(p)\in\Wc(p)$ such that
$\Dc(p')=(p,\varkappa(p))^c$.

\begin{lemma}\label{claim_boundary_pairing} For any point $p\in\Ws(a)$ we have
$\Ws(\varkappa(p))=\Ws(\varkappa(a))$ and the map $\Ws(a)\ni
p\mapsto\varkappa(p)\in\Ws(\varkappa(a))$ is a homeomorphism.
\end{lemma}

\begin{proof}
Continuity of the center foliation implies that every point $p\in
\Ws(a)$ has an open neighborhood $\V_p\subset\Ws(p)$ that maps
homeomorphically to an open neighborhood
$\V_{\varkappa(p)}\subset\Ws(\varkappa(p))$ of $\varkappa(p)$ via
$x\mapsto \varkappa(x)$. It follows that the collection of open sets
$$
\U_p=\{x\in \Ws(a)\colon\varkappa(a)=\varkappa(p)\}, \qquad
p\in\Ws(a)
$$
is a cover of $\Ws(a)$ by open disjoint sets (we ignore
repetitions). Since $\Ws(a)$ is connected we have that $\U_p=\Ws(a)$
for every $p\in\Ws(a)$ and $\varkappa(p)\in\Ws(\varkappa(a))$.

A similar argument shows that $\Ws(a)\ni
p\mapsto\varkappa(p)\in\Ws(\varkappa(a))$ is onto and continuity  of
$\Wc$ implies that it is a homeomorphism.
\end{proof}
Recall that the point $a\in\partial\Ws(\E)$ is uniquely defined by $y_0$. Consider the set
$$
\ss=\ss(y_0)\stackrel{\mathrm{def}}{=}\bigcup_{p\in\Ws(a)}(p, \varkappa(p))^c.
$$

It is clear that $\ss$ is an open subset of $\Ws(\E)$. Also it is
clear that $\ss$ is a trivial continuous fiber bundle over $\Ws(a)$.

\begin{lemma}\label{claim_center_lifting}
Given a path $\alpha\colon[0,1]\to\Ws(a)$ and a point
$z_0\in(\alpha(0), \varkappa(\alpha(0)))^c\cap\Ws(y_0)$ there exists
a unique lift $\tilde\alpha\colon[0,1]\to\Ws(y_0)$ such that
$\tilde\alpha(0)=z_0$ and $\tilde\alpha(t)=(\alpha(t),
\varkappa(\alpha(t)))^c$. Moreover, if $\alpha$ is a loop then
$\tilde\alpha$ is a loop as well.
\end{lemma}

\begin{proof}
The intersection
$\Ws(y_0)\cap\ss$ is an injectively immersed submanifold of $\ss$ without boundary which
is transverse to the fibers of $\ss$. This implies the existence and uniqueness of the lift
$\tilde\alpha$.

To prove the second part of the lemma recall that the set
$\Ws(z_0)\cap\Dc(z_0)$ can be ordered into a sequence $\{z_m,
m\in\Z\}$ as in the proof of Lemma~\ref{lemma_countable_covering} (for
$y_0$). Then $\tilde\alpha(1)\in\{z_m, m\in\Z\}$.

Assume that $\tilde\alpha(1)=z_k, k\ne 0$. For concreteness also assume
 that $k<0$.

The loop $(z_k, z_0)^c*\tilde\alpha$ can be perturbed to a smooth
loop $\beta\colon S^1\to \ss$ transverse to $\Ws$. Since the image
of $\beta$ is compact there exists a sufficiently large $R$ such
that $\beta(S^1)\subset\ss (R)$, where
$$
\ss(R)\stackrel{\mathrm{def}}{=}\bigcup_{p\in\Ws(a,R)}(p,\varkappa(p))^c.
$$
The set $\ss(R)$ is homeomorphic to a ball. Therefore, $\beta$ is
contractible inside $\Ws(\E)$. Since the image of this contraction is
compact we can choose sufficiently large $R'$ so that the above
contraction happens inside $\Ws(\E,R')$.

On the other hand, $\beta$ can be homotoped inside $\Ws(\E, R')$
to $\tau^{|k|}$ (recall that $\tau$ is the positive generator of
$\E$) by sliding the loop along the stable leaves towards $\E$.

The set $\Ws(\E, R')$ is a disk bundle over $\E$ and, hence, is
homotopy equivalent to $\E$. Therefore $\beta\sim\tau^{|k|}$ is
non-trivial in $\Ws(\E,R')$. We have arrived at a contradiction and
thus have established the second part of the lemma.
\end{proof}

Denote by $\U(y_k)$ the path connected component of
$\ss\cap\Ws(y_0)$ that contains $y_k$. Clearly, sets $\U(y_k)$,
$k\in \Z$, are open.

\begin{lemma}\label{claim_uk_exhaust}
Every point $z\in\ss\cap\Ws(y_0)$ belongs to  $\U(y_k)$ for some $k$.
\end{lemma}

\begin{proof}
Let $(p, \varkappa(p))^c$ be the center interval that contains $z$.
Let $\alpha\colon[0,1]\to\Ws(a)$ be a path that connects $p$ and $a$.
Then by Lemma~\ref{claim_center_lifting} $\alpha$ lifts to
$\tilde\alpha$ that connects $z$ and some point $y_k$.
\end{proof}

\begin{lemma} \label{claim_uk_disjoint}
Sets $\U(y_k)$, $k\in\Z$, are mutually disjoint.
\end{lemma}

\begin{proof}
Assume that $\U(y_i)=\U(y_j)$ with $i\ne j$. Then there is a path
$\tilde\alpha$ that connects $y_i$ and $y_j$ inside $\ss\cap\Ws(y_0)$.
Path $\tilde\alpha$ projects along the center fibers to a loop
$\alpha\colon S^1\to \Ws(a)$, $\alpha(0)=a$. But by
Lemma~\ref{claim_center_lifting} loop $\alpha$ lifts uniquely to a
loop in $\ss\cap\Ws(y_0)$. This gives us a contradiction since
$\tilde\alpha$ is a lift of $\alpha$ which is not a loop.
\end{proof}

\begin{lemma}\label{claim_uk_permuted}
The center holonomy $\H^c_{y_0}\colon\Ws(y_0)\to \Ws(y_0)$ cyclically
permutes sets $\U(y_k)$, $k\in\Z$, that is,
$\H^c_{y_0}(\U(y_k))=\U(y_{k+1})$, $k\in \Z$.
\end{lemma}

\begin{proof}
Recall that $\H^c_{y_0}(y_k)=y_{k+1}$. Consider any point $z\in\U(y_k)$
and a path $\tilde\alpha\colon[0,1]\to\U(y_k)$ that connects $y_k$ and
$z$. Then $\H^c_{y_0}\circ \tilde\alpha\colon[0,1]\to\Ws(y_0)$ is a
path that connects $y_{k+1}$ and $\H^c_{y_0}(z)$. Thus
$\H^c_{y_0}(\U(y_k))\subset \U(y_{k+1})$. A similar argument shows
that, in fact, $\H^c_{y_0}$ maps $\U(y_k)$ onto $\U(y_{k+1})$.
\end{proof}

\subsection{Wada Lakes structure on $\Ws(y_0)$}
\label{section_wada}
Wada Lakes are mutually disjoint connected open subsets of $\mathbb
R^d$ that share a common boundary. Clearly such open sets should have
rather weird shapes. It is possible to construct Wada Lakes through
an inductive procedure (see, \eg p. 143 of~\cite{HY}). Also Wada
Lakes appear naturally in various dynamical contexts (see,
\eg~\cite{C}).

Here we show that open sets $\U(y_k)$, $k\in\Z$, enjoy some
properties very similar to those of Wada Lakes.

The set $\{z\in \Ws(y_0)\colon\Dc(z)\ne\Wc(z)\}$ is open because of
continuity of the center foliation. Thus the set
$$
\K\stackrel{\mathrm{def}}{=}\{z\in \Ws(y_0)\colon\Wc(z)\subset\Ws(\E)\}
$$
is closed. Also note that $\K$ is not empty since $x\in\K$.

Given a point $z\in\K$ denote the multiplicity of the finite
covering projection $\H_{z\to \Ws(z)\cap\E}\colon\Ws(z)\to\E$ by $per(z)$.
Then, obviously, $(\H^c_{y_0})^{per(z)}(z)=z$. Therefore
$\H^c_{y_0}|_\K$ is a point-wise periodic homeomorphism.

\begin{lemma}\label{claim_dU_in_K}
For any $k\in\Z$ the boundary $\partial\U(y_k)$  of the
open set $\U(y_k)$ in $\Ws(y_0)$ is a subset of $\K$.
\end{lemma}

\begin{rmk}
Since $\U(y_k)\ne\Ws(y_0)$ it should have a non-empty boundary
$\partial\U(y_k)$ in $\Ws(y_0)$.
\end{rmk}

\begin{proof}
Pick a point $z\in\partial\U(y_k)$. If $z\not\in\K$ then it has a
small open connected neighborhood $\V(z)\subset\Ws(y_0)$ such that
for any point $z'\in\V(z)$ $\Dc(z')\neq\Wc(z')$.

Moreover, all the boundary points of $\Dc(z')$, $z'\in\V(z)$, lie on the
same two boundary components from $\partial\Ws(\E)$. (This can be seen from the argument in the proof of
Lemma~\ref{claim_boundary_pairing}.) Since the intersection $\V(z)\cap\U(y_k)$ 
is not empty the neighborhood $\V(z)$ must be a subset of
$\ss\cap\Ws(y_0)$; moreover, $\V(z)$ must be a subset of the connected
component of $y_k$, \ie  $\V(z)\subset\U(y_k)$. But this contradicts
to our assumption that $z\in\partial\U(y_k)$.
\end{proof}

\begin{lemma}\label{claim_uk_boundary_preserved}
Take any point $z\in\partial\U(y_k)$. Then $z\in\partial\U(y_{k+\ell
per(z)})$ for all $\ell\in\Z$.
\end{lemma}

\begin{proof}
By the previous lemma $z\in\K$ and thus $(\H^c_{y_0})^{\ell
per(z)}(z)=z$ for all $\ell \in\Z$. Recall that by
Lemma~\ref{claim_uk_permuted} $(\H^c_{y_0})^{\ell
per(z)}(\U(y_k))=\U(y_{k+\ell per(z)})$. The lemma follows.
\end{proof}

\begin{rmk}
We remark that it is not necessarily the case that $\Ws(y_0)$ is
decomposed as a disjoint union of $\K$ and $\U(y_k)$, $k\in\Z$.
Recall that the sequence $\{\U(y_k), k\in\Z\}$ is associated to a
pair of boundary components of $\Ws(\E)$: $\Ws(a)$ and $\Ws(\varkappa(a))$.
The boundary $\partial\Ws(\E)$ may have more, even infinitely many, 
components that can be paired in the same way and give rise to
other sequences of disjoint open sets inside $\Ws(y_0)$.

\end{rmk}

Only now we specialize to the codimension 2 case.
\begin{prop}\label{prop_C}
Let $\E$ be a length recurrent center leaf of a \ph \diff-sm $f$ that
satisfies assumptions of the Main Theorem. Assume that $\dim E^s=\dim E^u=2$. Also assume that $\Wc$ is
orientable. Then $\partial\W^\sigma(\E)=\varnothing$, $\sigma=s, u$.
\end{prop}

\begin{proof}
Assume that $\partial\Ws(\E)\ne \varnothing$. Then there exists a
point $y_0\in\Ws(\E)$ such that $\Dc(y_0)\ne\Wc(y_0)$ and all
constructions above apply. In particular, we have
the center holonomy $\H^c_{y_0}\colon \Ws(y_0)\to\Ws(y_0)$ and Wada
Lakes structure associated to it.

Consider a smooth loop $\alpha\colon S^1\to\Ws(x)$ that intersects
$\U(y_0)$ and passes through $x$. (Recall that
$x=\Ws(y_0)\cap\E$.) Then $\U(y_0)\cap\alpha(S^1)$ is a union of
disjoint open intervals. Let $\alpha((t,s))$ be one of these
intervals. Obviously
$\{\alpha(t),\alpha(s)\}\subset\partial\U(y_0)$. Therefore, by
Lemma~\ref{claim_dU_in_K}, $\alpha(t)$ and $\alpha(s)$ are in $\K$ and
must be fixed by some power of $\H^c_{y_0}$. Namely,
$$
(\H^c_{y_0})^{ per(\alpha(t))}(\alpha(t)) =\alpha(t) \quad \textup{and}
\quad  (\H^c_{y_0})^{per(\alpha(s))}(\alpha(s)) =\alpha(s).
$$
Consider the path $\beta\colon(t,s)\to\Ws(y_0)$ given by
$$
\beta(t')\stackrel{\mathrm{def}}{=}(\H^c_{y_0})^k(\alpha(t')), \qquad t'\in(t,s),
$$
where $k\stackrel{\mathrm{def}}{=}per(\alpha(t))per(\alpha(s))$. Obviously, $\beta$ has the
same endpoints as $\alpha$ and by Lemma~\ref{claim_uk_permuted}
$\beta((t,s))\in\U(y_k)$.

We concatenate $\alpha$ and $\beta$ to form a loop
$\ell\subset\Ws(y_0)$ half of which is in $\U(y_0)$ and the other
half in $\U(y_k)$.

Since $\Ws(y_0)$ is diffeomorphic to $\mathbb R^2$, the Jordan Curve Theorem
tells us that $\ell$ divides $\Ws(y_0)$ into the interior and
exterior parts.

We can find a point $z\in\partial\U(y_0)$ in the interior of $\ell$
and a point $\bar z\in\partial\U(y_0)$ in the exterior of $\ell$.
Indeed, just pick $z_0\in \ell\cap\U(y_0)$ and $z_1\in
\ell\cap\U(y_k)$ and connect $z_0$ and $z_1$ by a curve that lies in
the interior of $\ell$. Then point $z\in\partial\U(y_0)$ can be
found on this curve. Point $\bar z\in\partial\U(y_0)$ in the
exterior of $\ell$ can be found in a similar way.

Both $z$ and $\bar z$ are fixed by $(\H^c_{y_0})^{mper(z)per(\bar
z)}$, $m\in\Z$. Thus, by Lemma~\ref{claim_uk_permuted}, both $z$ and
$\bar z$ also belong to the boundaries of $\U(y_{mper(z)per(\bar
z)})$, $m\in \Z$. But this is impossible since $z$ and $\bar z$ lie
on different sides of $\ell$ and $\U(y_{mper(z)per(\bar z)})$ is
path connected and is different from $\U(y_0)$ and $\U(y_k)$ for a sufficiently large
$m$. We have arrived at a contradiction.
\end{proof}

\section{The proof of the Main Theorem}

We can assume that $\Wc$ is orientable. Otherwise we can pass to a double cover of $M$.

Let $\EuScript B$ be the bad set of $\Wc$. Assume that $\EuScript B\ne\varnothing$. It is
clear from the definition of $\EuScript B$ that $f\EuScript B=\EuScript B$.

Consider any $f$-invariant measure $\mu$ supported inside $\EuScript B$. Choose $C$
large enough so that the set
$$
A_C\stackrel{\mathrm{def}}{=}\{x\in\EuScript B: length(\Wc(x))\le C\}
$$
has positive $\mu$-measure. Then the Poincare Recurrence implies that there exists a length recurrent
center leaf $\E\subset A_C$.

Assume that $\dim E^s=2$. Then the discussion in Section~\ref{section_boundary} applies to
$\widehat\W^s(\widehat \E)$. In particular, given $x\in\widehat \E$ we have the center holonomy
$$
\H_x^c\colon\widehat\W^s(x)\to\widehat\W^s(x).
$$
By Proposition~\ref{prop_C} $\widehat\W^s(\widehat\E)$ does not have accessible boundary and hence the leaves $\widehat\W^c(y)$, $y\in\widehat\W^s(x)$, are complete circles embedded in $\widehat \W^s(\widehat\E)$. Then Lemma~\ref{lemma_finite_covering} implies that every $y\in\widehat\W^s(x)$ is a periodic 
point of $\H_x^c$.

Thus $\H_x^c$ is a point-wise periodic \homeo of $\widehat\W^s(x)$. By a classical theorem of Montgomery~\cite{M}
$\H_x^c$ must be periodic. This means that $\widehat \E$ has a small foliated neighborhood $V$ and that the holonomy group of $\widehat\E$ in $\widehat\W^s(\widehat\E)$ (that is, $G_x^s(\widehat\W^c)$, $x\in\widehat\E$) is finite. Then $p(V)$
(map $p\colon \widehat\W^s(\widehat\E)\to\W^s(\E)$ was defined in Section~\ref{section_reduction}) is a foliated neigborhood of $\E$ in $\W^s(\E)$ and the holonomy group $G_{p(x)}^s(\Wc)$ is also finite.

If $\dim E^s=1$ then the arguments from Section~\ref{section_boundary} are not needed. One can see that $G_x^s(\Wc)$, $x\in\E$,
has order $1$ or $2$ (the latter in the case when $\Ws$ is not orientable) by repeating the simple argument of 
Lemma~\ref{lemma_codim_1}.

Similarly we see that $G_x^u(\Wc)$, $x\in\E$, is also finite. Thus, by Proposition~\ref{prop_product_holonomy},
the full holonomy group is a finite group of order $k$.

Then by the Generalized Reeb Stability Theorem $\E$ has a small foliated neighborhood $U$ such that
every center leaf in $U$ covers $\E$ at most $k$ times. This implies that
$$
\forall z\in U \;\;\; length(\Wc(z))\le k\cdot length(\E)+K
$$ 
for some $K>0$.
But $\E$ is in the bad set so $length$ must be locally unbounded at $\E$. We have arrived at a contradiction. Hence
the bad set $\EuScript B$ is empty and the Main Theorem follows.
\begin{remark}
\begin{enumerate}
\item[]
\item The assumption $\dim E^s\le 2$, $\dim E^u\le 2$ is only used in the proof of Proposition~\ref{prop_C}.
Thus the Main Theorem would immediately generalize to the higher dimensional setup if one can show that Wada Lakes structure
described in Section~\ref{section_wada} cannot exist on leaves of dimension $\ge3$ as well.
\item Accessible boundary appears naturally whenever one looks at classification problems for 
\ph diffeomorphism. See, \eg~\cite{BI, BW}. It is also implicit in the proofs of Theorems~\ref{TA} and~\ref{TB}.
\end{enumerate}
\end{remark}

\end{document}